\documentclass[11pt,reqno]{article} 
\usepackage[latin1]{inputenc}   
\usepackage{amsfonts,amsmath,layout}
\usepackage{mathrsfs}  
\usepackage{amssymb,mathabx,amsfonts,amsthm,amscd,stmaryrd,dsfont,esint,upgreek,constants,todonotes}
\DeclareFontFamily{OT1}{pzc}{}
\DeclareFontShape{OT1}{pzc}{m}{it}{<-> s * [1.10] pzcmi7t}{}
\DeclareMathAlphabet{\mathpzc}{OT1}{pzc}{m}{it}

\usepackage{color} 

\definecolor{Red}{cmyk}{0,1,1,0.2}
\usepackage{hyperref}

\newcommand{\N}{\mathbb N}

\newcommand{\R}{\mathbb R}

\def\R{\mathbb R}
\def\N{\mathbb N}

\def\E{\mathbb E}
\def\P{\mathbb P}

\def\ep{\epsilon}

\def\dk{{\bf d}_1}

\newcommand{\be}{\begin{equation}}
\newcommand{\ee}{\end{equation}}
\def\1{{\bf 1}}

\def\dive{{\rm div}}

\def\Pw{{\mathcal P}_2(\R^d)}
\def\Pk{{\mathcal P}_1(\R^d)}

\def\ds{\displaystyle}

\newtheorem{Theorem}{Theorem}[section]
\newtheorem{Definition}[Theorem]{Definition}
\newtheorem{Proposition}[Theorem]{Proposition}
\newtheorem{Lemma}[Theorem]{Lemma}

\begin{document}


\title{Remarks on Nash equilibria in mean field game models with a major player}
\author{P.~Cardaliaguet\thanks{Universit\'e Paris-Dauphine, PSL Research University, CNRS, Ceremade, 75016 Paris, France - cardaliaguet@ceremade.dauphine.fr } \and M.~Cirant\thanks{Dipartimento di Matematica "Tullio Levi-Civita", Università di Padova via Trieste 63, 35121 Padova (Italy) - cirant@math.unipd.it} \and A.~Porretta\thanks{Dipartimento di Matematica, Universit\`a di Roma ``Tor Vergata'', Via della Ricerca Scientifica 1, 00133 Roma, Italy - porretta@mat.uniroma2.it}}

\maketitle

\begin{abstract} For a mean field game model with a major and infinite minor players, we characterize a notion of Nash equilibrium via a system of so-called master equations, namely a system of nonlinear transport equations in the space of measures. Then, for games with a finite number~$N$ of minor players and a major player, we prove that the solution of
the corresponding Nash system converges to the solution of the system of master equations as $N$ tends to infinity.
\end{abstract}



\section*{Introduction}

The aim of this note is to discuss the model of a mean field game (MFG) with a major and many minor players. Let us recall that MFGs describe Nash equilibrium configurations in differential games with infinitely many small players. 

MFG problems with a major player are differential games in which infinitely many small players interact with a major one. This class of problems  was first introduced by Huang in \cite{H10} and later studied in various forms and  different frameworks in several papers \cite{BCY15, BCY16, BLP14, CK14, CDbook, CW16, CW17, CZ16, EMP, FC, KP, LL18, MB15, NH12, NC13, SC16}. In the literature, the notion of solution is often that of Nash equilibria and, in the present paper, we will concentrate on this notion of solution. Let us point out however that this is not always the case in the above quoted references: for instance, \cite{BCY15, BCY16} (see also \cite{MB15}) study Stackelberg equilibria (in which the major player announces in advance his strategy); \cite{EMP} deals with the (closely related) problem of principal-agents; \cite{BLP14} considers the situation where the small agents cooperate to play a zero-sum game against the major player. 

Concerning  Nash equilibria in the framework of MFG with a major player, Carmona and  Zhu \cite{CZ16} point out that the notion is rather subtle and not  trivial at all. In fact, \cite{CZ16} (and, subsequently,  \cite{CDbook} and \cite{CW17}) propose  a notion of Nash equilibrium which differs from the classical construction of \cite{NH12, NC13} (see Proposition 6.2 of \cite{CZ16}).  In the same papers, it is proved by the authors that this definition yields $\varepsilon$-Nash equilibria for the $N+1$-players' game in the linear quadratic case. Very recently, Lasry and Lions \cite{LL18}  introduce a new equation (the master equation with a major player, see equation \eqref{eq.MasterMM} below) which could give rise to yet another notion of solution. Finally, one could also wonder if the limit of games with finitely many players (including a major player) as  the number of small players tends to infinity would not give rise again to a different notion of solution. The purpose of this paper is to show that the approach by Carmona and al.  \cite{CDbook, CW17, CZ16}, the master equation of Lasry and Lions \cite{LL18} and the limit of Nash equilibria, as the number of small players tends to infinity, lead to the same  Nash equilibria.

To explain our result, let us start with the differential game with $N$ minor players and one major player, which we describe through the following PDE system; the value functions associated with the  $N$ minor players are denoted by $u^{N,i}$ ($i=1,\dots,N$) while the value function of the major player is $u^{N,0}$. To simplify the discussion, we assume that the players have only individual noises. The Nash system reads (in $(0,T)\times \R^{d_0+Nd}$) (see Section \ref{sec.not} below for the notations): 
\be \label{NashSyst}
\left\{\begin{array}{l}
\ds -\partial_t u^{N,0}-\sum_{j=0}^N \Delta_{x_j} u^{N,0} + H^0(x_0,D_{x_0}u^{N,0}, m^{N}_{\bf x})
+ \sum_{j =1}^N D_{x_j}u^{N,0}\cdot D_pH(x_j,x_0, D_{x_j}u^{N,j}, m^{N, j}_{\bf x})=0 \\ 
\ds -\partial_t u^{N,i}-\sum_{j=0}^N \Delta_{x_j} u^{N,i} + H(x_i,x_0,D_{x_i}u^{N,i}, m^{N,i}_{\bf x}) \\
\ds \qquad  + D_{x_0}u^{N,i} \cdot D_pH^0(x_0,D_{x_0}u^{N,0}, m^{N}_{\bf x})
+ \sum_{j\neq i,\ j\geq 1} D_{x_j}u^{N,i}\cdot D_pH(x_j,x_0, D_{x_j}u^{N,j}, m^{N, j}_{\bf x})=0 \\ 
u^{N,0}(T,{\bf x})= G^0(x_0,m^{N}_{\bf x}), \; u^{N,i}(T, {\bf x})= G(x_i,x_0,m^{N,i}_{\bf x}).
\end{array}\right.
\ee
where ${\bf x}= (x_0, \dots, x_N)$, $\ds m^{N}_{\bf x}=\frac{1}{N}\sum_{i\geq 1} \delta_{x_i}$, $\ds m^{N,i}_{\bf x}=\frac{1}{N-1}\sum_{j\neq \{0,i\}} \delta_{x_j}$. 

 Following \cite{friedman1972stochastic} for instance, the solution $(u^{N,0}, u^{N,1}, \dots, u^{N,N})$ describes the payoff at equilibrium of a $(N+1)-$player stochastic differential game. The particular structure of this system expresses the fact that the ``small" players (for $i=1, \dots, N$) have a symmetric cost function, giving rise to the same Hamiltonian $H$. In addition, for a small player $i=1, \dots, N$, the other players are indistinguishable and appear only through the empirical measure $m^{N,i}_{\bf x}$ in the Hamiltonian $H$ of this player. In the same way, the small players are indistinguishable for the large player: they appear through the empirical measure $ m^{N}_{\bf x}$ in the Hamiltonian $H^0$ of the large player. The fact that the players differ in size is expressed by the fact that the position of a small player $i$ (for $i=1, \dots, N$) enters in the Hamiltonian $H^0$ of the major player and in the Hamiltonian $H$ of the other small players with a weight $1/N$ or $1/(N-1)$, while  the position of the major player (i.e., for $i=0$) enters 
in the Hamiltonian $H$ the other players without weight.

Because of the symmetry of system \eqref{NashSyst} and the uniqueness of its solution, one can check that  $u^{N,0}$  only depends on $(t,x_0)$ and on the empirical measure $m^N_{\bf x}$ of the small players while $u^{N,i}$  depends on $(t, x_i, x_0)$ and on 
the empirical measure $m^{N,i}_{\bf x}$ (player $i$ playing a particular role for $v^{N,i}$). 

So, arguing as in \cite{CDLL}, one formally expects, as the number $N$ of small players tends to infinity, that $u^{N,0}\sim U^0(t,x_0,m^{N}_{\bf x})$, $u^{N,i}\sim U(t,x_i, x_0, m^{N,i}_{\bf x})$, where $U^0, U$ solve the system of master equations:  
\be\label{eq.MasterMM}
\left\{\begin{array}{rl}
\ds (i)& \ds -\partial_t U^{0}-\Delta_{x_0} U^0+ H^0(x_0,D_{x_0}U^{0}, m) -\int_{\R^d} \dive_yD_m U^0(t,x_0,m,y)dm(y)\\
&\ds \qquad + \int_{\R^d} D_mU^0(t,x_0,m,y)\cdot D_pH(y,x_0, D_{x}U(t,y,x_0,m), m)dm(y)=0 \\ 
&\ds \hspace{8cm} {\rm in} \; (0,T)\times \R^{d_0}\times {\mathcal P_2}(\R^d),\\
\ds (ii) & \ds -\partial_t U- \Delta_{x} U-\Delta_{x_0}U+ H(x,x_0,D_{x}U, m) -\int_{\R^d} \dive_y D_mU(t,x,x_0,m,y)dm(y)\\
&\ds \qquad  + D_{x_0}U \cdot D_pH^0(x_0,D_{x_0}U^0(t,x_0,m), m)\\
&\ds \qquad   + \int_{\R^d} D_m U(t,x,x_0,m,y)\cdot D_pH(y,x_0, D_{x}U(t,y,x_0,m), m)dm(y)=0 \\ 
&\ds \hspace{8cm} {\rm in} \; (0,T)\times \R^d\times \R^{d_0}\times {\mathcal P_2}(\R^d),\\
\ds (iii) & \ds U^{0}(T,x_0,m)= G^0(x_0,m), \qquad  {\rm in} \; \R^{d_0}\times {\mathcal P_2}(\R^d),\\
\ds (iv) & \ds U(T,x, x_0,m)= G(x,x_0,m)\qquad  {\rm in} \; \R^d\times \R^{d_0}\times {\mathcal P_2}(\R^d).\\
\end{array}\right.
\ee
This is a nonlinear equation stated in the space of probability measures $\Pw$ of $\R^d$. The notion of derivative with respect to a measure used in the above system is the same as in \cite{CDLL}. The master equation (without a major agent) was first introduced by Lasry and Lions and discussed by Lions in \cite{LLperso}. It was later studied in \cite{CDLL, CCP, CDbook, CgCrDe, GS14-2} in various degrees of generality. 
System \eqref{eq.MasterMM} is precisely the system (here without common noise, to simplify the expressions) introduced  in \cite{LL18}. The well-posedness in short time of \eqref{eq.MasterMM} is stated in \cite{LL18},  a  detailed proof  (with a completely different approach) is also contained in our companion paper \cite{CCP}. The first goal of the present paper is to show, through a classical verification argument, that \eqref{eq.MasterMM} yields  the notion of Nash equilibria  (in the case of Markovian feedback controls) introduced by Carmona and al. in \cite{CDbook, CW17, CZ16}: see Proposition \ref{prop.verification}. 
Then we rigorously prove that the solution of the Nash system \eqref{NashSyst} converges to the solution of the system of master equations \eqref{eq.MasterMM} as the number of players tends to infinity: see Theorem \ref{theo.limit}. The main interest of this result is that it provides another justification of the definition. The method of proof  follows closely the lines of \cite{CDLL} (see also \cite{CDbook}), where a similar statement for problems without a major player but with common noise is established.

\section{Notation and assumptions}\label{sec.not}

The state space of the major player is $\R^{d_0}$ ($d_0\in \N$, $d_0\geq 1$), the state space for the minor player is $\R^d$ ($d\in \N$, $d\geq 1$). Both spaces are endowed with the Euclidean distance $|\cdot|$. 

We denote by ${\mathcal P}(\R^d)$ the set of Borel probability measures on $\R^d$ and by  ${\mathcal P}_k(\R^d)$, $k\geq 1$, the set of measures in ${\mathcal P}(\R^d)$ with finite moment of order $k$: namely,
$$
M_k(m):= \left(\int_{\R^d} |x|^k m(dx)\right)^{1/k} <+\infty \qquad {\rm if} \; m\in {\mathcal P}_k(\R^d). 
$$
The set ${\mathcal P}_k(\R^d)$ is endowed with the distance  (see for instance \cite{AGS, RaRu98, villani})
$$
{\bf d}_k(m,m') =\inf_{\pi} \left( \int_{\R^d} |x-y|^k\pi(dx,dy)\right)^{1/k}, \qquad \forall m,m'\in {\mathcal P}_k(\R^d), 
$$
where the infimum is taken over the couplings $\pi$ between $m$ and $m'$, i.e., over the Borel probability measures $\pi$ on $\R^d\times \R^d$ with first marginal $m$ and second marginal $m'$. 

Given a map $U:\Pw\to \R$, the notion of derivative in the space of measures is the one introduced in \cite{CDLL} and  \cite{CCP}.  We say that a map $U:\Pw\to \R$ is $C^1$ is there exists a {\it continuous and bounded} map $\frac{\delta U}{\delta m}:\Pw\times \R^d\to \R$ such that 
$$
U(m')-U(m)= \int_0^1 \int_{\R^d} \frac{\delta U}{\delta m}((1-s)m+sm',y) (m'-m)(dy)ds \qquad \forall m,m'\in \Pw. 
$$
We say that the map $U$ is continuously $L-$differentiable if $U$ is $C^1$ and if $y\to \frac{\delta U}{\delta m}(m,y)$ is everywhere differentiable with a  continuous and globally bounded derivative on $\Pw\times \R^d$. We   denote by
\be\label{DM}
D_mU(m,y):= D_y \frac{\delta U}{\delta m}(m,y)
\ee
this $L-$derivative.

The Hamiltonians of the problem are $H^0: \R^{d_0}\times \R^{d_0}\times \Pw\to\R$ for the major player and $H:\R^d\times \R^{d_0}\times \R^d\times \Pw\to\R$ for the minor players. We assume that $H^0$ and $H$ are smooth enough to justify the computations below. In particular, $H^0$ and $H$ are assumed to be at least $C^1$ with respect to the measure variable. The maps $p^0\to H^0(x_0,p^0,m)$ and $p\to H(x,x_0,p,m)$ are also assumed to be strictly convex: 
$$
D^2_{p^0p^0}H^0(x_0,p^0,m)>0, \; D^2_{pp}H(x,x_0,p,m)>0.
$$
We denote by $L^0$ and $L$ the convex conjugate of $H^0$ and $H$ with respect to the variable $p^0$ and $p$ respectively: 
$$
L^0(x_0,\alpha^0, m)= \sup_{p^0\in \R^{d_0}} - \alpha^0\cdot p^0-H^0(x_0,p^0,m)
$$
and
$$
L(x,x_0,\alpha,m)= \sup_{p\in \R^d} -\alpha\cdot p -H(x,x_0,p,m).
$$

\section{Interpretation of the model}
 
In this section, we show that the system \eqref{eq.MasterMM} of master equations can be interpreted as a Nash equilibrium in the infinite players' game with a major player. 

Namely we consider the infinite players' game in which the minor players play  closed loop strategies of the form $\alpha=\alpha(t,x,x_0, m)$ (where $x$ is the position of the minor player, $x_0$ is the position of the major player and $m$ is the distribution of the minor players), while the major player plays a closed loop strategy of the form $\alpha^0=\alpha^0 (t,x_0,m)$.   Here and below,  $\alpha(t,x,x_0,m)$ and $\alpha^0(t,x^0,m)$ are (deterministic) functions which we assume to be bounded and locally Lipschitz continuous on, respectively, $[0,T]\times \R^d\times \R^{d_0}\times \Pw$  and $[0,T]\times  \R^{d_0}\times \Pw$. Thus, both  the representative minor agent and the major agent are playing feedback strategies. In particular, given a  Brownian motion $\{B^0_t\}$ in $\R^{d_0}$, and a stochastic flow of measures $\{m_t\}$ in $\Pw$ which is adapted to the filtration  generated by  $B^0:=\{B^0_t\}$,  the dynamics of the major player will be given by
$$
d  X^0_t =  \alpha^0(t,  X^0_t,  m_t)dt +\sqrt{2}dB^0_t
$$
while the dynamics of the representative minor player is given by
$$
d  X_t =  \alpha (t, X_t, X^0_t,  m_t)dt +\sqrt{2}dB_t\,,
$$
where $B:=\{B_t\}$ is a Brownian motion independent of $B^0$. We stress that the choice of Markovian feedback controls  implies that  the stochastic controls $\alpha^0_t= \alpha^0(t,  X^0_t, m_t)$, and, respectively, $\alpha_t=\alpha(t, X_t, X^0_t, m_t)$ are adapted to the filtrations generated by $(B^0_t)$ and, respectively, by $(B^0_t, B_t)$.

Notice that the stochastic flow of measures $m_t$, which is going to represent the distribution of the minor agents, must necessarily be adapted to the filtration $B^0$ of the major player. At equilibrium, $m_t$ will be the conditional law given $B^0_t$ of $X_t$. 

The following definition is a reformulation of the definition of \cite{CW17}  by using the stochastic  PDE satisfied by the distribution law $m_t$: 

\begin{Definition} Given an initial measure $ \mu_0\in \Pw$ and an initial position $x_0^0\in \R^{d_0}$ for the major player, a Nash equilibrium in the game is a pair $(\bar \alpha, \bar \alpha^0)$ of  feedback strategies for the minor and major player with the following properties:
\begin{enumerate}

\item $(\bar X^0_t, \bar m_t)$ are the  flow of positions for the major player and of the mean field for the minor players generated by $\bar \alpha$ and $\bar \alpha^0$, i.e.  the solution to the McKean-Vlasov  stochastic system:
\be\label{eq:McV}
\left\{ \begin{array}{l}
d\bar X^0_t = \bar \alpha^0(t,\bar X^0_t, \bar m_t)dt +\sqrt{2}dB^0_t\qquad {\rm in}\; [0,T]\\
d_t\bar   m_t= \left\{ \Delta\bar  m_t  - \dive(\bar m_t \bar \alpha(t,x,\bar X^0_t,\bar m_t))\right\}dt  \qquad {\rm in}\; [0,T]\times \R^d,\\
\bar m_0= \mu_0, \; \bar X^0_0=x_0^0. 
\end{array}\right.
\ee
\item The feedback strategy $\bar \alpha$ is optimal for each minor player, given $(\bar X^0_t, \bar m_t)$, namely  
\be\label{eqcondminor}
  J(\bar \alpha ; \bar \alpha^0, \bar m_t) \leq J(\alpha; \bar \alpha^0,  \bar m_t)  
\ee
 for any Markovian feedback control $\alpha_t:=\alpha(t,X_t,\bar X^0_t,\bar m_t)$, where
$$
J(\alpha; \bar \alpha^0, \bar m_t) =\E\left[ \int_0^T L(X_t,\bar X^0_t,  \alpha(t,X_t,\bar X^0_t, \bar m_t) , \bar m_t) dt + G(X_T, \bar X^0_T, \bar m_T)\right], 
$$
with $dX_t = \alpha_t dt +\sqrt{2}dB_t$, $X_0$ being distributed according to $\mu_0$. 

\item  The feedback strategy  $\bar \alpha^0$ is optimal for the major player, meaning that 
$$
J^0(\bar \alpha; \bar \alpha^0) \leq J^0(\bar \alpha; \alpha^0), 
$$
 for any different feedback law $\alpha^0(t,x,m)$, where 
$$
J^0(\bar \alpha; \alpha^0)= \E\left[ \int_0^T L^0(X^0_t,\alpha^0_t(t,X^0_t , m_t) ,  m_t) dt + G^0(X^0_T,  m_T)\right], 
$$
where $(X^0_t,m_t)$ is now the  flow of positions for the major player and of the mean field for the minor players  generated by $\bar \alpha$ and $ \alpha^0$, i.e., the solution to
\be\label{VerifXm}
\left\{ \begin{array}{l}
d X^0_t =  \alpha^0(t, X^0_t,m_t)dt +\sqrt{2}dB^0_t \qquad {\rm in}\; [0,T]\\
d_t m_t= \left\{ \Delta m_t  -\dive(m_t \bar \alpha(t,x, X^0_t,m_t))\right\}dt\qquad {\rm in}\; [0,T]\times \R^d,\\
m_0=\mu_0, \; X^0_0=x_0^0. 
\end{array}\right.
\ee
\end{enumerate}
\end{Definition}

A few comments on the definition are now in order. We first note that  $\bar m_t$ (and, respectively, $m_t$) are nothing but the conditional expectation given $(\bar X^0_s)_{s\leq t}$ of  the process  $\bar X_t$ (respectively, the conditional expectation given $(X^0_s)_{s\leq t}$ of  $X_t$), where $\bar X_t$, $X_t$ are solutions of the McKean-Vlasov SDEs
$$
d\bar X_s= \bar \alpha(s,\bar X_s, \bar X^0_s,  {\mathcal L}(\bar X_s|\bar X^0_s))ds + \sqrt{2}dB_s, \qquad {\mathcal L}(\bar X_0)=\mu_0, 
$$
and, respectively,
$$
dX_s=  \bar  \alpha(s,X_s,  X^0_s,  {\mathcal L}( X_s|X^0_s))ds + \sqrt{2}dB_s, \qquad {\mathcal L}(X_0)=\mu_0\,. 
$$
 We stress that, given any couple $\bar \alpha^0(t,x,m), \bar \alpha(t,x,x^0,m)$ of bounded and locally Lipschitz functions, the existence of  a (unique) solution $(\bar X_t, \bar m_t )$ of the system \eqref{eq:McV} can be proved with standard fixed point methods.

 The asymmetry between the major and the infinitely many minor players  appears clearly in the above definition of Nash equilibrium. Indeed, the cost for a minor player who deviates playing a strategy $\alpha$ is $J(\alpha ; \bar \alpha^0, \bar m_t)$ because, for this deviating small player, the mean field is fixed (since the strategies of the other minor players are fixed) as well as the corresponding  strategy of  the major player. In contrast, if the major player deviates, the mean field $(m_t)$ also changes because all the minor players react to the deviation. 
Note that this definition is exactly the one introduced by Carmona-Wang~ \cite{CW17} in their ``closed loop version"  of Section~2.  In fact, as detailed in \cite{CZ16}, the equilibrium defined so far can also be interpreted as a Nash equilibrium of a  two-player differential game by first defining  the cost of the small player  as $J(\alpha; \alpha^0, m_t)$ for an exogenous stochastic flow of measures and  the cost of the major player as $J^0(\alpha; \alpha^0)$ as above, and then requiring that the flow $m_t$  satisfies, at the Nash equilibrium, the consistency condition $m_t={\mathcal L}( \bar X_s|\bar X^0_s)$.

We now show the link between the  system of master equations and the Nash equilibria of the MFG problem with a major agent. 

\begin{Proposition}[Verification]\label{prop.verification} Let $(U^0,U)$ be a classical solution to the system of master equations \eqref{eq.MasterMM}. Then the pair
$$
(\bar \alpha(t,x,x_0,m), \bar \alpha^0(t,x_0,m)):= -(D_pH(x,x_0, D_xU(t,x,x_0,m),m), D_pH^0(x_0, D_{x_0}U^0(t,x_0,m),m))
$$
is a Nash equilibrium of the game. 
\end{Proposition}

\begin{proof} Let us first check that \eqref{eqcondminor} holds. For any $\alpha$, we have, in view of the equation satisfied by $(\bar m_t)$,  
\begin{align*}
dU(t,X_t,\bar X^0_t, \bar m_t) & = \Bigl\{ \partial_t U+ D_x U\cdot \alpha(t,X_t,\bar X^0_t, \bar m_t) + D_{x_0} U\cdot \bar \alpha^0(t,\bar X^0_t, \bar m_t) +\Delta_x U+ \Delta_{x_0} U\\
& \qquad  +\int_{\R^d} \dive_y (D_mU(t,X_t,\bar X^0_t, \bar m_t,y))\bar m_t(dy) \\
& \qquad  -\int_{\R^d} D_mU(t,X_t,\bar X^0_t, \bar m_t,y) \cdot D_pH(y,\bar X^0_t, D_xU(t,y,\bar X^0_t, \bar m_t), \bar m_t)\bar m_t (dy)\Bigr\}dt \\
& \qquad + \sqrt{2}\left( D_xU \cdot dB_t+ D_{x_0}U\cdot dB^0_t\right),
\end{align*}
where, unless otherwise specified, $U$ and its space derivatives are evaluated at $(t,X_t,\bar X^0_t, \bar m_t)$. Using the equation satisfied by $U$ and the definition of $\bar \alpha^0$, we obtain after integration in time and taking expectation:  
\begin{align*}
& \E\left[ G(X_T, \bar X^0_T, \bar m_T)\right] = \E\left[ U(T,X_T,\bar X^0_T, \bar m_T)\right] \\
& \qquad \qquad  = \E\left[ U(0, X_0, x^0_0, \mu_0)\right]+\int_0^T \E\left[ D_x U\cdot \alpha + H(X_t,\bar X^0_t, D_xU(t,X_t,\bar X^0_t, \bar m_t),\bar m_t)\right]dt\\
& \qquad \qquad  \geq  \E\left[ U(0, X_0, x^0_0, \mu_0)\right]-\int_0^T \E\left[ L(X_t,\bar X^0_t, \alpha(t,X_t,\bar X^0_t, \bar m_t), \bar m_t)\right]dt,
\end{align*}
with an equality if 
$$
\alpha(t,X_t,\bar X^0_t, \bar m_t)=-D_pH(t,X_t,\bar X^0_t, D_xU(t,X_t,\bar X^0_t, \bar m_t), \bar m_t)=
\bar \alpha(t,X_t,\bar X^0_t, \bar m_t). 
$$
This shows that 
$$
J(\alpha; \bar X^0_t, \bar m_t) \geq \E\left[ U(0, X_0, x^0_0, \mu_0)\right] =J(\bar \alpha; \bar X^0_t, \bar m_t), 
$$ 
so that $\bar \alpha$ is optimal. 

Next we show the optimality of $\bar \alpha^0$. Let $\alpha^0$ be a feedback for the major player, $(X_t, m_t)$ be given by \eqref{VerifXm}. Then 
\begin{align*}
dU^0(t, X^0_t, m_t) & = \Bigl\{ \partial_t U^0+  D_{x_0} U^0\cdot \alpha^0 +\Delta_{x_0} U^0
  +\int_{\R^d} \dive_y (D_mU^0(t, X^0_t,  m_t,y)) m_t(dy) \\
& \qquad  -\int_{\R^d} D_mU^0(t, X^0_t,  m_t,y) \cdot D_pH(y,  X^0_t, D_xU(t,y,X^0_t, m_t),  m_t) m_t (dy)\Bigr\}dt \\
& \qquad + \sqrt{2} D_{x_0}U\cdot dB^0_t,
\end{align*}
where, unless otherwise specified, $U^0$ and its space derivatives are evaluated at $(t,X^0_t,  m_t)$. Therefore, in view of the equation satisfied by $U^0$, we have 
\begin{align*}
\E\left[ G^0(X^0_T,  m_T)\right]& = \E\left[ U^0(T,X^0_T,  m_T)\right] \\
&  = U^0(0, x_0, \mu_0)+\int_0^T \E\left[ D_{x_0} U^0\cdot \alpha^0 + H^0(X^0_t, D_{x_0}U^0(t,X^0_t, m_t), m_t)\right]dt\\
& \geq U^0(0, x_0, \mu_0)-\int_0^T \E\left[ L^0(X^0_t, \alpha^0(t,X^0_t,  m_t), m_t)\right]dt,
\end{align*}
with an equality if 
$$
\alpha^0(t,X^0_t,  m_t)=-D_pH^0(t, X^0_t, D_{x_0}U(t,X^0_t, m_t),  m_t)=
\bar \alpha^0(t,X^0_t,  m_t),
$$
in which case $m=\bar m$. 
This shows that 
$$
J^0(\bar \alpha; \alpha^0 ) \geq  U^0(0, x_0, \mu_0) = J^0(\bar \alpha;\bar \alpha^0 )
$$ 
and proves the optimality of $\bar \alpha^0$. 
\end{proof}

\section{The mean field limit}

In this part we show that the master equation \eqref{eq.MasterMM} corresponds to the mean field limit of the $N-$player game with a major player. 
We work here with the $\dk$ distance and we assume that $H^0$, $D_pH^0$, $H$ and $D_pH$ are globally Lipschitz continuous in the sense that (for instance for $H$)
$$
|H(x,x_0, p, m)-H(x',(x_0)', p', m')| \leq C (|x-x'|+|x_0-(x_0)'|+|p-p'|+\dk(m,m'))
$$
for any $x,x'\in \R^d$, $x_0,x_0'\in \R^{d_0}$, $p,p'\in \R^d$, $m,m'\in \Pk$. 
Our main result is the following:

\begin{Theorem}\label{theo.limit} Let $(u^{N,i})$ be a classical solution to the Nash system \eqref{NashSyst} and $(U^0,U)$ be a classical solution to the system \eqref{eq.MasterMM} of master equations. There is a constant $C$, independent of $N$, ${\bf x}\in \R^{d_0}\times (\R^d)^{N}$ and $t\in [0,T]$, such that
$$
\left| u^{N,0}(t,{\bf x})- U^0(t, x_0, m^N_{\bf x}) \right| + 
\sup_{i=1,\dots, N} \left| u^{N,i}(t,{\bf x}) - U(t, x_i,x_0, m^{N,i}_{\bf x})\right| \leq CN^{-1} \left(1+ \frac{1}{N}\sum_{i=1}^N |x_i|\right),
$$
where, as before,  
$$
m^N_{\bf x}= \frac{1}{N} \sum_{i=1}^N \delta_{x_i}, \qquad m^{N,i}_{\bf x}= \frac{1}{N-1} \sum_{j\notin \{0, i\} } \delta_{x_j}. 
$$
\end{Theorem}

As in \cite{CDLL}, it is also possible to show that the optimal trajectories associated with the $N-$player problem converge to the optimal trajectory for the limit one.

\begin{proof} 
We  follow the strategy of proof of \cite{CDLL}.
Let $(u^{N,i})$ be the solution to \eqref{NashSyst} and $(U^0, U)$ be the solution of  \eqref{eq.MasterMM}. Following \cite{CDLL}, we set 
$$
v^{N,0}(t,{\bf x})= U^0(t, x_0, m^N_{\bf x}), \; v^{N,i}(t,{\bf x})= U(t,x_i, x_0, m^{N,i}_{\bf x}). 
$$
Let us fix $(t_0,x_0^0,\mu_0)\in [0,T]\times \R^{d_0}\times \Pw$ and let  $(Z^{N,i})_{i\geq 1}$ be i.i.d. random variables with law $\mu_0\in \Pw$. 
We consider the system ${\bf X}= (X^{N,0}, X^{N,1}, \cdots, X^{N,N})$
 of SDEs: 
\begin{align*}
& dX^{N,0}_s= -D_pH^0(X^{N,0}_s, D_{x_0} u^{N,0}(s, {\bf X}_s),m^N_{{\bf X}_s})ds+ \sqrt{2} dB^0_s ,\qquad s\in [ t_0,T], \\
& dX^{N,i}_s= -D_pH(X^{N,i}_s, X^{N,0}_s,D_{x_i} u^{N,i}(s, {\bf X}_s),m^{N,i}_{{\bf X}_s})ds+ \sqrt{2} dB^i_s ,\qquad s\in [ t_0,T], \\
& X^{N,0}_{t_0}= x_0^0, \; X^{N,i}_{t_0} = Z^{N,i}.
\end{align*}

 Let us first notice that the $(v^{N,i})$ are almost solutions to the Nash system: 

\begin{Lemma}\label{lem.vNi} For $i=0, \dots, N$, there exist continuous maps $r^{N,i}:[0,T]\times \R^{d_0+Nd}\to \R$ such that 
$$
\left\{\begin{array}{l}
\ds -\partial_t v^{N,0}-\sum_{j=0}^N \Delta_{x_j} v^{N,0} + H^0(x_0,D_{x_0}v^{N,0}, m^{N}_{\bf x})
+ \sum_{j\geq 1} D_{x_j}v^{N,0}\cdot D_pH(x_j,x_0, D_{x_j}v^{N,j}, m^{N, j}_{\bf x})=r^{N,0} \\ 
\ds -\partial_t v^{N,i}-\sum_{j=0}^N \Delta_{x_j} v^{N,i} + H(x_i,x_0,D_{x_i}v^{N,i}, m^{N,i}_{\bf x}) \\
\ds \qquad  + D_{x_0}v^{N,i} \cdot D_pH^0(x_0,D_{x_0}v^{N,0}, m^{N}_{\bf x})
+ \sum_{j\neq i,\ j\geq 1} D_{x_j}v^{N,i}\cdot D_pH(x_j,x_0, D_{x_j}v^{N,j}, m^{N, j}_{\bf x})=r^{N,i} \\ 
v^{N,0}(T,{\bf x})= G^0(x_0,m^{N}_{\bf x}), \; v^{N,i}(T, {\bf x})= G(x_i,x_0,m^{N,i}_{\bf x}).
\end{array}\right.
$$
and 
$$
\sup_{i=0, \dots, N} \|r^{N,i}(t,{\bf x})\|_\infty\leq 
\frac{C}{N} (1+M_1(m^N_{\bf x})).
$$
\end{Lemma}

Recall that $M_1(m)= \int_{\R^d} |x|m(dx)$ is the first order moment of the measure $m$.  We postpone the proof of the Lemma and proceed with the ongoing proof. The main part of the proof consists in estimating the difference $u^{N,i}-v^{N,i}$ along the trajectory ${\bf X}$. For this we set 
$$
U^{N,i}_t = u^{N,i}(t,{\bf X}^{N}_t), \qquad V^{N,i}_t = v^{N,i}(t, {\bf X}^{N}_t). 
$$
We first note that, in view of the equation satisfied by  the $(u^{N,i})$, 
\begin{align*}
dU^{N,0}_t &= \Bigl(\partial_t u^{N,0} + \sum_{j\geq 0} \Delta_{x_j} u^{N,0}   -D_pH^0(X^{N,0}_t, D_{x_0}u^{N,0}(t, {\bf X}_t),m^N_{{\bf X}_t})\cdot D_{x_0}u^{N,0} \\
 & \qquad
-\sum_{j\geq 1} D_pH(X^{N,j}_t, X^{N,0}_t,D_{x_j}u^{N,j}, m^{N,j}_{{\bf X}_t})\cdot D_{x_j} u^{N,0}\Bigr)dt 
+\sqrt{2} \sum_{j\geq 0} D_{x_j}u^{N,0}\cdot dB^j_t \\
& =  (H^0(X^{N,0}_t,D_{x_0}u^{N,0},m^N_{{\bf X}_t})    -D_pH^0(X^{N,0}_t, D_{x_0}u^{N,0},m^N_{{\bf X}_t})\cdot D_{x_0}u^{N,0} \\
& \qquad 
+\sqrt{2} \sum_{j\geq 0} D_{x_j}u^{N,0}\cdot dB^j_t,
\end{align*}
where the $u^{N,j}$ are evaluated at $(t,{\bf X}^{N}_t)$. 
On the other hand, by Lemma \ref{lem.vNi}, we have 
\begin{align*}
d V^{N,0}_t &= (\partial_t v^{N,0} + \sum_{j\geq 0} \Delta_{x_j} v^{N,0}   -D_pH^0(X^{N,0}_t, D_{x_0}u^{N,0}(t, {\bf X}_t),m^N_{{\bf X}_t})\cdot D_{x_0}v^{N,0} \\
 & \qquad
-\sum_{j\geq 1} D_pH(X^{N,j}_t, X^{N,0}_t,D_{x_j}u^{N,j}, m^{N,j}_{{\bf X}_t})\cdot D_{x_j} v^{N,0})dt 
+\sqrt{2} \sum_{j\geq 0} D_{x_j}v^{N,0}\cdot dB^j_t \\
&= \Bigl( H^0(X^{N,0}_t, D_{x_0}v^{N,0}, m^N_{{\bf X}_t}) -D_pH^0(X^{N,0}_t, D_{x_0}u^{N,0}(t, {\bf X}_t),m^N_{{\bf X}_t})\cdot D_{x_0}v^{N,0} \\
 & \qquad
-\sum_{j\geq 1} (D_pH(X^{N,j}_t, X^{N,0}_t,D_{x_j}u^{N,j}, m^{N,j}_{{\bf X}_t})
-D_pH(X^{N,j}_t, X^{N,0}_t, D_{x_j}v^{N,j}, m^{N,j}_{{\bf X}_t})) \cdot D_{x_j} v^{N,0} \\
& \qquad    - r^{N,0} \Bigr)dt 
+\sqrt{2} \sum_{j\geq 0} D_{x_j}v^{N,0}\cdot dB^j_t ,
\end{align*}
where the $u^{N,j}$, $r^{N,0}$ and $v^{N,0}$ are, here again, evaluated at $(t,{\bf X}_t)$. So, for any $s\in [t_0,T]$, 
\begin{align*}
& (U^{N,0}_T-V^{N,0}_T)^2  = (U^{N,0}_{s}-V^{N,0}_{s})^2 \\
&  +
\int_{s}^T 2(U^{N,0}_{t}-V^{N,0}_{t}) \Bigl(H^0(X^{N,0}_t,D_{x_0}u^{N,0},m^N_{{\bf X}_t})-H^0(X^{N,0}_t, D_{x_0}v^{N,0}, m^N_{{\bf X}_t}) \\
& \qquad -D_pH^0(X^{N,0}_t, D_{x_0}u^{N,0},m^N_{{\bf X}_t})\cdot (D_{x_0}u^{N,0} -D_{x_0}v^{N,0}) \\
& \qquad 
+ \sum_{j\geq 1} (D_pH(X^{N,j}_t, X^{N,0}_t,D_{x_j}u^{N,j}, m^{N,j}_{{\bf X}_t})
-D_pH(X^{N,j}_t, X^{N,0}_t, D_{x_j}v^{N,j}, m^{N,j}_{{\bf X}_t})) \cdot D_{x_j} v^{N,0}  \\
&\qquad +r^{N,0} \Bigr)dt
+2 \sum_{j\geq0} \int_{s}^T | D_{x_j}u^{N,0}-D_{x_j}v^{N,0}|^2 dt  \\ 
& \qquad \qquad 
+ 2\sqrt{2}  \sum_{j\geq 0} \int_{s}^T (U^{N,0}_{t}-V^{N,0}_{t}) (D_{x_j}u^{N,0}-D_{x_j}v^{N,0})\cdot dB^j_t .
\end{align*}
Note that $U^{N,0}_T=V^{N,0}_T$ because the maps  $u^{N,0}$ and $v^{N,0}$ have the same terminal condition.  
Using the global Lipschitz continuity of $H^0$ and $H$ and the fact that $\|D_{x_j}v^{N,0}\|_\infty\leq CN^{-1}$ for $j\geq 1$, we infer that 
\begin{align*}
& 0  \geq (U^{N,0}_{s}-V^{N,0}_{s})^2 \\
&  - C
\int_{s}^T |U^{N,0}_{t}-V^{N,0}_{t}| \Bigl( |D_{x_0}(u^{N,0}-v^{N,0})|  
+ N^{-1} \sum_{j\geq 1} |D_{x_j}(u^{N,j}-v^{N,j})|+|r^{N,0}| \Bigr)dt \\
&  + 2\sum_{j\geq0} \int_{s}^T | D_{x_j}(u^{N,0}-v^{N,0})|^2 dt  
+ 2\sqrt{2}  \sum_{j\geq 0} \int_{s}^T (u^{N,0}-v^{N,0}) D_{x_j}(u^{N,0}-v^{N,0})\cdot dB^j_t .
\end{align*}
Taking the conditional expectation $ \E^{\bf Z}$ given  ${\bf Z}$ and using Young's inequality we find, for any $\ep>0$,  
\begin{align}
& 0  \geq \E^{\bf Z}\left[(U^{N,0}_{s}-V^{N,0}_{s})^2\right] 
- \int_{s}^T  \E^{\bf Z}\Bigl[  C\ep^{-1} |U^{N,0}_{t}-V^{N,0}_{t}|^2  +\ep |D_{x_0}(u^{N,0}-v^{N,0})|^2  +\ep |r^{N,0}|^2 \notag
\\
&\qquad\qquad\qquad + C\ep N^{-1} \sum_{j\geq 1} |D_{x_j}(u^{N,j}-v^{N,j})|^2 \Bigr]dt \notag\\
& \qquad + 2\sum_{j\geq0} \int_{s}^T   \E^{\bf Z}\Bigl[ | D_{x_j}(u^{N,0}-v^{N,0})|^2\Bigr] dt  . \label{khqebsdn}
\end{align}
Note that the estimate of $r^{N,i}$ in Lemma \ref{lem.vNi} implies that 
$$
\E^{{\bf Z}}[(r^{N,0}(t, {\bf X}^N_t))^2] \leq \frac{C}{N^2}(1+ \E^{{\bf Z}}[M_1(m^N_{{\bf X}^N_t})])^2, 
$$
where, in view of the uniform  bound on $D_pH^0$ and $D_pH$, we have 
$$
\E^{{\bf Z}}[M_1(m^N_{{\bf X}^N_t})]\leq C(1+M_1(m^N_{{\bf Z}})). 
$$
So
$$
\E^{{\bf Z}}[(r^{N,0}(t, {\bf X}^N_t))^2] \leq \frac{C}{N^2}(1+ M_1(m^N_{{\bf Z}}))^2= \frac{C_{{\bf Z}}}{N^2}, 
$$
where $C_{{\bf Z}}:= C(1+ M_1(m^N_{{\bf Z}}))^2$. Coming back to \eqref{khqebsdn}, we find, for  $\ep$ small, 
\begin{align}
& C_{{\bf Z}} \ep N^{-2}  \geq  \E^{\bf Z}\left[(U^{N,0}_{s}-V^{N,0}_{s})^2\right] \notag\\
& \qquad \qquad 
- C\int_{s}^T  \E^{\bf Z}\Bigl[ \ep^{-1} |U^{N,0}_{t}-V^{N,0}_{t}|^2 
+ \ep N^{-1} \sum_{j\geq 1} |D_{x_j}(u^{N,j}-v^{N,j})|^2 \Bigr]dt \label{unun}\\
& \qquad\qquad + \sum_{j\geq0} \int_{s}^T   \E^{\bf Z}\Bigl[ | D_{x_j}(u^{N,0}-v^{N,0})|^2\Bigr] dt  . \notag
\end{align}

We now make the same computation for $i\geq 1$. We have 
\begin{align*}
dU^{N,i}_t &= \Bigl(\partial_t u^{N,i} + \sum_{j\geq 0} \Delta_{x_j} u^{N,i}   -D_pH^0(X^{N,0}_t, D_{x_0}u^{N,0},m^N_{{\bf X}_t})\cdot D_{x_0}u^{N,i} \\
 & \qquad
-\sum_{j\geq 1} D_pH(X^{N,j}_t, X^{N,0}_t,D_{x_j}u^{N,j}, m^{N,j}_{{\bf X}_t})\cdot D_{x_j} u^{N,i}\Bigr)dt 
+\sqrt{2} \sum_{j\geq 0} D_{x_j}u^{N,i}\cdot dB^j_t \\
& =  (H(X^{N,i}_t,X^{N,0}_t,D_{x_i}u^{N,i},m^{N,i}_{{\bf X}_t})  -D_pH(X^{N,i}_t, X^{N,0}_t, D_{x_i}u^{N,i}, m^{N,i}_{{\bf X}_t})\cdot D_{x_i} u^{N,i} \Bigr)dt\\
& \qquad    
+\sqrt{2} \sum_{j\geq 0} D_{x_j}u^{N,i}\cdot dB^j_t,
\end{align*}
where the $u^{N,j}$ is evaluated at $(t,{\bf X}_t)$. 
On the other hand, by Lemma \ref{lem.vNi}, we have 
\begin{align*}
d V^{N,i}_t &= \Bigl(\partial_t v^{N,i} + \sum_{j\geq 0} \Delta_{x_j} v^{N,i}   -D_pH^0(X^{N,0}_t, D_{x_0}u^{N,0},m^N_{{\bf X}_t})\cdot D_{x_0}v^{N,i} \\
 & \qquad
-\sum_{j\geq 1} D_pH(X^{N,j}_t, X^{N,0}_t, D_{x_j}u^{N,j}, m^{N,j}_{{\bf X}_t})\cdot D_{x_j} v^{N,i}\Bigr)dt 
+\sqrt{2} \sum_{j\geq 0} D_{x_j}v^{N,i}\cdot dB^j_t \\
&= \Bigl( H(X^{N,i}_t,X^{N,0}_t, D_{x_i}v^{N,i}, m^N_{{\bf X}^N_t})
- D_pH(X^{N,i}_t, X^{N,0}_t, D_{x_i}u^{N,i}, m^{N,i}_{{\bf X}_t})\cdot D_{x_i} v^{N,i} \\
& \qquad -(D_pH^0(X^{N,0}_t, D_{x_0}u^{N,0},m^N_{{\bf X}_t})
-D_pH^0(X^{N,0}_t, D_{x_0}v^{N,0},m^N_{{\bf X}_t})) \cdot D_{x_0}v^{N,i} \\
 & \qquad
-\sum_{j\neq \{0,i\}} (D_pH(X^{N,j}_t, X^{N,0}_t, D_{x_j}u^{N,j}, m^{N,j}_{{\bf X}_t})
-D_pH(X^{N,j}_t, X^{N,0}_t, D_{x_j}v^{N,j}, m^{N,j}_{{\bf X}_t})) \cdot D_{x_j} v^{N,i} \\
& \qquad    - r^{N,i} \Bigr)dt 
+\sqrt{2} \sum_{j\geq 0} D_{x_j}v^{N,i}\cdot dB^j_t \ ,
\end{align*}
where the $u^{N,j}$, $r^{N,j}$ and $v^{N,j}$ are, here again, evaluated at $(t,{\bf X}_t)$. So, for any $s\in [t_0,T]$, 
\begin{align*}
& (U^{N,i}_T-V^{N,i}_T)^2  = (U^{N,i}_{s}-V^{N,i}_{s})^2 \\
&  +
\int_{s}^T 2(U^{N,i}_{t}-V^{N,i}_{t}) \Bigl(H(X^{N,i}_t,X^{N,0}_t,D_{x_i}u^{N,i},m^N_{{\bf X}_t})-H(X^{N,i}_t,X^{N,0}_t, D_{x_i}v^{N,i}, m^N_{{\bf X}_t}) \\
& \qquad -D_pH(X^{N,i}_t, X^{N,0}_t, D_{x_i}u^{N,i}, m^{N,i}_{{\bf X}_t})\cdot D_{x_i} (u^{N,i}-v^{N,i}) \\
& \qquad -(D_pH^0(X^{N,0}_t, D_{x_0}u^{N,0},m^N_{{\bf X}_t})
-D_pH^0(X^{N,0}_t, D_{x_0}v^{N,0},m^N_{{\bf X}_t})) \cdot D_{x_0}v^{N,i} \\
& \qquad 
+ \sum_{j\neq 0, i} (D_pH(X^{N,j}_t, X^{N,0}_t,D_{x_j}u^{N,j}, m^{N,j}_{{\bf X}_t})
-D_pH(X^{N,j}_t, X^{N,0}_t, D_{x_j}v^{N,j}, m^{N,j}_{{\bf X}_t})) \cdot D_{x_j} v^{N,i}  \\
&\qquad - r^{N,i} \Bigr)dt
+ 2\sum_{j\geq0} \int_{s}^T | D_{x_j}(u^{N,i}-v^{N,i})|^2 dt  
\\ 
& \qquad \qquad 
+ 2\sqrt{2}  \sum_{j\geq 0} \int_{s}^T (U^{N,i}_{t}-V^{N,i}_{t}) D_{x_j}(u^{N,i}-v^{N,i})\cdot dB^j_t .
\end{align*}
Since  $U^{N,i}_T=V^{N,i}_T$ and $H^0$ and $H$ are Lipschitz continuous and since $\|D_{x_j}v^{n,i}\|_\infty\leq CN^{-1}$ for $j\neq \{0,i\}$, we have
\begin{align*}
& 0  \geq (U^{N,i}_{s}-V^{N,i}_{s})^2 \\
&  -
\int_{s}^T C|U^{N,i}_{t}-V^{N,i}_{t}| \Bigl( |D_{x_i}(u^{N,i}-v^{N,i})|  + |D_{x_0}(u^{N,0}-v^{N,0})|  \\
& \qquad\qquad\qquad\qquad\qquad\qquad + N^{-1} \sum_{j\neq \{0,i\}} |D_{x_j}(u^{N,j}-v^{N,j}|+|r^{N,i}| \Bigr)dt \\
&  +2\sum_{j\geq0} \int_{s}^T | D_{x_j}(u^{N,i}-v^{N,i})|^2 dt  
+ 2\sqrt{2}  \sum_{j\geq 0} \int_{s}^T (u^{N,i}-v^{N,i}) D_{x_j}(u^{N,i}-v^{N,i})\cdot dB^j_t .
\end{align*}
Taking expectation and using Young's inequality and the  estimate of $r^{N,i}$ in Lemma \ref{lem.vNi}, we find, for $\ep$ small enough,  
\begin{align}
& C_{{\bf Z}}\ep N^{-2}  \geq  \E^{\bf Z}\left[(U^{N,i}_{s}-V^{N,i}_{s})^2\right] 
- C \int_{s}^T  \E^{\bf Z}\Bigl[  \ep^{-1} |U^{N,i}_{t}-V^{N,i}_{t}|^2    \notag \\
&\qquad\qquad +\ep |D_{x_0}(u^{N,0}-v^{N,0})|^2 + \ep N^{-1} \sum_{j\neq \{0,i\}} |D_{x_j}(u^{N,j}-v^{N,j})|^2 \Bigr]dt \label{deuxdeux} \\
&\qquad \qquad + \sum_{j\geq0} \int_{s}^T   \E^{\bf Z}\Bigl[ | D_{x_j}(u^{N,i}-v^{N,i})|^2\Bigr] dt . \notag 
\end{align}
We add  inequalities in \eqref{deuxdeux} for $i=1,\dots, N$ with N times inequality \eqref{unun} to obtain
\begin{align}
& C_{{\bf Z}}\ep N^{-1}  \geq  N\E^{\bf Z}\left[(U^{N,0}_{s}-V^{N,0}_{s})^2\right]+\sum_{i\geq 1} \E^{\bf Z}\left[(U^{N,i}_{s}-V^{N,i}_{s})^2\right] \notag \\
& \qquad  
- CN \ep^{-1} \int_{s}^T \E^{\bf Z}\left[   |U^{N,0}_{t}-V^{N,0}_{t}|^2 \right]dt + C\ep^{-1} \sum_{i\geq 1}  \int_{s}^T \E^{\bf Z}\left[   |U^{N,i}_{t}-V^{N,i}_{t}|^2\right] dt
\notag \\
&\qquad 
- CN \ep \int_s^T \E^{\bf Z}\left[ |D_{x_0}(u^{N,0}-v^{N,0})|^2 \right]dt
-  C\ep \sum_{j\geq 1} \int_s^T \E^{\bf Z}\left[ |D_{x_j}(u^{N,j}-v^{N,j})|^2 \right]dt \label{khjzqbeslkndc}\\
& \qquad +N \sum_{j\geq0} \int_{s}^T   \E^{\bf Z}\Bigl[ | D_{x_j}(u^{N,0}-v^{N,0})|^2\Bigr] dt 
 + \sum_{i\geq 1, j\geq0} \int_{s}^T   \E^{\bf Z}\Bigl[ | D_{x_j}(u^{N,i}-v^{N,i})|^2\Bigr] dt .\label{khjzqbeslkndc2}
\end{align}
Choosing a last time $\ep$ small enough, we can absorb the terms in line \eqref{khjzqbeslkndc} into the term in line \eqref{khjzqbeslkndc2}. 
Then, by Gronwall's Lemma, we find
\begin{align}
& N\E^{\bf Z}\left[(U^{N,0}_{s}-V^{N,0}_{s})^2\right]+\sum_{i\geq 1} \E^{\bf Z}\left[(U^{N,i}_{s}-V^{N,i}_{s})^2\right]  \label{kuhzbaefd}\\
& +N \sum_{j\geq0} \int_{s}^T   \E^{\bf Z}\Bigl[ | D_{x_j}(u^{N,0}-v^{N,0})|^2\Bigr] dt 
 + \sum_{i\geq 1, j\geq0} \int_{s}^T   \E^{\bf Z}\Bigl[ | D_{x_j}(u^{N,i}-v^{N,i})|^2\Bigr] dt \notag
 \leq C_{{\bf Z}}N^{-1}.
\end{align}
We use this inequality to evaluate the second line in \eqref{deuxdeux}: for $i=1, \dots, N$ we have
\begin{align*}
& C_{{\bf Z}} N^{-2}  \geq  \E^{\bf Z}\left[(U^{N,i}_{s}-V^{N,i}_{s})^2\right] 
- C\ep^{-1} \int_{s}^T  \E^{\bf Z}\Bigl[   |U^{N,i}_{t}-V^{N,i}_{t}|^2 \Bigr]dt    \\
&\qquad \qquad + \sum_{j\geq0} \int_{s}^T   \E^{\bf Z}\Bigl[ | D_{x_j}(u^{N,i}-v^{N,i})|^2\Bigr] dt   
\end{align*}
and finally obtain, thanks again to Gronwall's Lemma: 
\begin{align*}
& \E^{\bf Z}\left[(U^{N,i}_{s}-V^{N,i}_{s})^2\right] 
+\sum_{j\geq0} \int_{s}^T   \E^{\bf Z}\Bigl[ | D_{x_j}(u^{N,i}-v^{N,i})|^2\Bigr] dt  \leq C_{{\bf Z}} N^{-2} . 
\end{align*}
For $s=t_0$ and in view of the initial condition of the process ${\bf X}$, this proves that, $\P-$a.s. and for any $i=1, \dots, N$, 
\begin{align*}
& |u^{N,i}(t_0,{\bf Z})-v^{N,i}(t_0,{\bf Z})|^2=  \E^{\bf Z}\left[(U^{N,i}_{t_0}-V^{N,i}_{t_0})^2\right] 
\leq C_{{\bf Z}} N^{-2}   
\end{align*}
while, for  $i=0$, we have by \eqref{kuhzbaefd}:
\begin{align*}
& |u^{N,0}(t_0,{\bf Z})-v^{N,0}(t_0,{\bf Z})|^2=  \E^{\bf Z}\left[(U^{N,0}_{t_0}-V^{N,0}_{t_0})^2\right] 
\leq C_{{\bf Z}} N^{-2} . 
\end{align*}
If we choose the $Z_i$ identically distributed with a positive density and finite first order moment, we obtain the Theorem by the continuity of $u^{N,i}$ and of $U$. 
\end{proof}

\begin{proof}[Proof of Lemma \ref{lem.vNi}]  Let us recall the following relations, proved in \cite[Proposition 6.1.1]{CDLL}. For $i,j \neq 0$ with $i\neq j$, we have
\begin{align*}
& D_{x_0} v^{N,0} (t,{\bf x})= D_{x_0} U^0(t,x_0, m^N_{\bf x}), \; D_{x_i} v^{N,0} (t,{\bf x})= \frac{1}{N} D_m U^0(t,x_0, m^N_{\bf x},x_i), \\
& D^2_{x_ix_j} v^{N,0} (t,{\bf x})= \frac{1}{N^2} D^2_{mm} U^0(t,x_0, m^N_{\bf x},x_i,x_j), \\ 
& D^2_{x_ix_i} v^{N,0} (t,{\bf x})= \frac{1}{N^2} D^2_{mm} U^0(t,x_0, m^N_{\bf x},x_i,x_i)+ \frac{1}{N} D^2_{ym}   U^0(t,x_0, m^N_{\bf x},x_i),
\end{align*}
 where, we recall that $D_m U^0$ depends on one extra variable (see \eqref{DM}) and consequently, $D_{mm}^2U^0$ depends on  two extra variables.
 
The corresponding equalities hold for $v^{N,i}$ with $i\geq 1$ (with $1/N$ replaced by $1/(N-1)$). 
So 
\begin{align*}
&-\partial_t v^{N,0}-\sum_{j=0}^N \Delta_{x_j} v^{N,0} + H^0(x_0,D_{x_0}v^{N,0}, m^{N}_{\bf x}) 
+ \sum_{j\geq 1} D_{x_j}v^{N,0}\cdot D_pH(x_j,x_0, D_{x_j}v^{N,j}, m^{N, j}_{\bf x}) \\ 
& = -\partial_t U^{0}(t,x_0, m^N_{\bf x})- \Delta_{x_0} U^0(t,x_0, m^N_{\bf x})\\
& \qquad - \sum_{j=1}^N {\rm Tr} \left(\frac{1}{N^2} D^2_{mm} U^0(t,x_0, m^N_{\bf x},x_i,x_i)+ \frac{1}{N} D^2_{ym}   U^0(t,x_0, m^N_{\bf x},x_i)\right) \\
& \qquad + H^0(x_0,D_{x_0}U^{0}(t,x_0, m^N_{\bf x}), m^{N}_{\bf x}) \\ 
& \qquad + \frac{1}{N} \sum_{j\geq 1} D_{m}U^0  (t,x_0, m^N_{\bf x},x_j)
\cdot D_pH(x_j,x_0, D_{x}U(t,x_j,x_0, m^{N,j}_{\bf x}), m^{N, j}_{\bf x}) \\
& = -\partial_t U^{0}(t,x_0, m^N_{\bf x})- \Delta_{x_0} U^0(t,x_0, m^N_{\bf x}) + H^0(x_0,D_{x_0}U^{0}(t,x_0, m^N_{\bf x}), m^{N}_{\bf x}) \\ 
& \qquad +\int_{\R^d}  D_{m}U^0  (t,x_0, m^N_{\bf x},y) \cdot D_pH(y,x_0, D_{x_j}U(t,y,x_0, m^{N}_{\bf x}), m^{N}_{\bf x})m^N_{\bf x}(dy)  \\
& \qquad -\int_{\R^d} \dive_y D_mU^0 (t,x_0, m^N_{\bf x},y) m^N_{\bf x}(dy) +r^{N,0}(t,{\bf x}) =r^{N,0}(t,{\bf x})
\end{align*}
thanks to the equation satisfied by $U^0$, where 
\begin{align*}
 r^{N,0}(t,{\bf x}) & := -\sum_{j=1}^N  \frac{1}{N^2} {\rm Tr} D^2_{mm} U^0(t,x_0, m^N_{\bf x},x_i,x_i) 
- \frac{1}{N} \sum_{j\geq 1} D_{m}U^0  (t,x_0, m^N_{\bf x},x_j) \times
\\
& 
\Big[D_pH(x_j,x_0, D_{x_j}U(t,x_j,x_0, m^{N,j}_{\bf x}), m^{N, j}_{\bf x}) 
-D_pH(x_j,x_0, D_{x_j}U(t,x_j,x_0, m^{N}_{\bf x}), m^{N}_{\bf x})  \Big]. 
\end{align*}
Note that 
\begin{align*}
\frac{1}{N} \sum_{j\geq 1}  \dk( m^{N,j}_{\bf x}, m^{N}_{\bf x}) & 
\leq  \frac{1}{N^2(N-1)} \sum_{j\geq 1} \sum_{i\neq 0, j}  |x_i-x_j| \leq \frac{C}{N} M_1(m^N_{\bf x}).
\end{align*}
So, by the regularity of $U^0$, $U$ and $H$, we have 
$$
|r^{N,0}(t,{\bf x})| \leq \frac{C}{N} (1+M_1(m^N_{\bf x})),
$$
as claimed. 

The proof for $v^{N,i}$ goes along the same lines and we omit it. 
\end{proof}

\end{document}